\documentclass{amsart}

\theoremstyle{definition}

\theoremstyle{remark}

\reversemarginpar

\begin{document}

\title[Elementary logarithmic family]{The 
integrals in Gradshteyn and Rhyzik. Part 2: \\
Elementary logarithmic integrals}

\author{Victor H. Moll}
\address{Department of Mathematics,
Tulane University, New Orleans, LA 70118}
\email{vhm@math.tulane.edu}

\subjclass{Primary 33}

\date{\today}


\begin{abstract}
We describe methods to evaluate elementary logarithmic integrals. The integrand
is the product of a rational function and a linear polynomial in $\ln x$. 
\end{abstract}

\maketitle

\newcommand{\nn}{\nonumber}
\newcommand{\ba}{\begin{eqnarray}}
\newcommand{\ea}{\end{eqnarray}}
\newcommand{\ift}{\int_{0}^{\infty}}
\newcommand{\ione}{\int_{0}^{1}}
\newcommand{\ifft}{\int_{- \infty}^{\infty}}
\newcommand{\no}{\noindent}
\newcommand{\Ftwo}{{{_{2}F_{1}}}}
\newcommand{\realpart}{\mathop{\rm Re}\nolimits}
\newcommand{\imagpart}{\mathop{\rm Im}\nolimits}

\newtheorem{Definition}{\bf Definition}[section]
\newtheorem{Thm}[Definition]{\bf Theorem} 
\newtheorem{Example}[Definition]{\bf Example} 
\newtheorem{Lem}[Definition]{\bf Lemma} 
\newtheorem{Note}[Definition]{\bf Note} 
\newtheorem{Cor}[Definition]{\bf Corollary} 
\newtheorem{Prop}[Definition]{\bf Proposition} 
\newtheorem{Problem}[Definition]{\bf Problem} 
\numberwithin{equation}{section}

\section{Introduction} \label{intro} 
\setcounter{equation}{0}

The table of integrals by I. M. Gradshteyn and I. M. Rhyzik \cite{gr} 
contains a large selection of definite integrals of the form
\begin{equation}
\int_{a}^{b} R(x) \ln^{m}x \, dx, 
\label{basic}
\end{equation}
\noindent
where $R(x)$ is a rational function, $a, \, b \in \mathbb{R}^{+}$ 
 and $m \in \mathbb{N}$. We call integrals 
of the form (\ref{basic}) {\em elementary logarithmic integrals}.
The goal of this
note is to present methods to evaluate them. We may assume that $a=0$ using 
\begin{equation}
\int_{a}^{b} R(x) \ln^{m}x \, dx = \int_{0}^{b} R(x) \, \ln^{m}x \, dx - 
\int_{0}^{a} R(x) \, \ln^{m}x \, dx. 
\end{equation}

Section \ref{sec-poly} describes the situation when $R$ is a polynomial. 
Section \ref{sec-linear} presents the case in which the rational function 
has a single simple pole. Finally section \ref{sec-multiple} considers the
case of multiple poles.

\section{Polynomials examples} \label{sec-poly} 
\setcounter{equation}{0}

The first example considered here is
\begin{equation}
I(P;b,m) := \int_{0}^{b} P(x) \, \ln^{m}x \, dx,
\end{equation}
\noindent
where $P$ is a polynomial. This can 
be evaluated in elementary terms. Indeed, 
$I(P;b,m)$ is a linear combination of 
\begin{equation}
\int_{0}^{b} x^{j}  \, \ln^{m}x \, dx, 
\end{equation}
\noindent
and the change of variables $x = bt$ yields
\begin{equation}
\int_{0}^{b} x^{j} \ln^{m}x \, dx = 
b^{j+1} \sum_{k=0}^{m} \binom{m}{k} \ln^{m-k}b \int_{0}^{1} t^{j} \ln^{k}t \, 
dt. 
\end{equation}
\noindent
The last integral evaluates to  $(-1)^{k}k!/(j+1)^{k+1}$ either an easy 
induction argument or by the change of variables $t = e^{-s}$ that gives
it as a value of the gamma function. 

\begin{Thm}
Let $P(x)$ be a polynomial given by 
\begin{equation}
P(x) = \sum_{j=0}^{p}a_{j}x^{j}. 
\end{equation}
\noindent
Then
\begin{equation}
I(P;b,m) := \int_{0}^{b} P(x) \ln^{m}x \, dx = 
\sum_{k=0}^{m} (-1)^{k} k! \binom{m}{k} \ln^{m-k}b \,
\sum_{j=0}^{p} a_{j} \frac{b^{j+1}}{(j+1)^{k+1}}. 
\end{equation}
\noindent
This expression shows that $I(P;b,m)$ is a linear combination of 
$b^{j} \, \ln^{k}b$, with $1 \leq j \leq 1 + p ( = 1 + \text{deg}(P))$ and 
$0 \leq k \leq m$. 
\end{Thm}

\section{Linear denominators} \label{sec-linear} 
\setcounter{equation}{0}

We now consider the integral 
\begin{equation}
f(b;r) := \int_{0}^{b} \frac{\ln x \, dx}{x + r}
\end{equation}
\noindent
for $b, \, r > 0$. This corresponds to the case in which the rational function 
in (\ref{basic}) has a single simple pole. 

The change of variables $x = rt$ produces
\begin{equation}
\int_{0}^{b} \frac{\ln x \, dx}{x + r} = 
\ln r \, \ln (1 + b/r) + \int_{0}^{b/r} \frac{\ln t \, dt}{1+t}. 
\label{case1}
\end{equation}
\noindent
Therefore, it suffices to consider the function
\begin{equation}
g(b) := \int_{0}^{b} \frac{\ln t \, dt}{1+t},
\end{equation}
\noindent
as we have
\begin{equation}
f(b;r) = \ln r \, \ln \left( 1 + \frac{b}{r} \right) + g \left( \frac{b}{r}
\right). 
\end{equation}

Before we present a discussion of the function $g$, we describe some 
elementary consequences of (\ref{case1}).  \\

\noindent
{\bf Elementary examples}. The special case $r=b$ in (\ref{case1}) yields
\begin{equation}
\int_{0}^{b} \frac{dx}{x+b} = \ln 2 \, \ln b + \ione \frac{\ln t \, dt}{1+t}. 
\end{equation}
\noindent
Expanding $1/(1+t)$ as a geometric series, we obtain
\begin{equation}
\ione \frac{\ln t \, dt}{1+t} = - \frac{1}{2} \zeta(2) = - \frac{\pi^{2}}{12}.
\end{equation}
\noindent
This appears as $\mathbf{4.231.1}$ in \cite{gr}. Differentiating
(\ref{case1}) with respect to $r$ produces
\begin{equation}
\int_{0}^{b} \frac{\ln x \, dx }{(x+r)^{2}} = 
-\frac{\ln(b+r)}{r} + \frac{\ln r}{r} + \frac{b \, \ln b}{r(r+b)}. 
\label{case3}
\end{equation}
\noindent
As $b, \, r \to 1$ we obtain 
\begin{equation}
\int_{0}^{1} \frac{\ln x \, dx}{(1+x)^{2}} = - \ln 2.
\end{equation}
\noindent
This appears as $\mathbf{4.231.6}$ in \cite{gr}. On the other hand, as $b \to \infty$
we recover $\mathbf{4.231.5}$ in \cite{gr}: 
\begin{equation}
\ift \frac{\ln x \, dx}{(x+r)^{2}} = \frac{\ln r}{r}. 
\end{equation}

\medskip
\noindent
{\bf The polylogarithm function}. The evaluation of the integral
\begin{equation}
g(b) := \int_{0}^{b} \frac{\ln t \, dt}{1+t},
\end{equation}
\noindent
requires the transcendental function
\begin{equation}
\text{Li}_{n}(x) := \sum_{k=1}^{\infty} \frac{x^{k}}{k^{n}}.
\end{equation}
\noindent
This is the {\em polylogarithm function} and it has also appeared in 
\cite{moll-gr1} in our discussion of the  family
\begin{equation}
h_{n}(a) := \ift \frac{\ln^{n}x \, dx}{(x-1)(x+a)}, \quad n \in \mathbb{R}, \,
a > 0. 
\end{equation}
\noindent
In the current context we have $n=2$ and we are dealing with the 
{\em dilogarithm function}: $\text{Li}_{2}(x)$. 

\medskip

\begin{Lem}
The function $g(b)$ is given by
\begin{equation}
g(b) = \ln b \, \ln(1+b) + \text{Li}_{2}(-b).
\end{equation}
\end{Lem}
\begin{proof}
The change of variables $t = bs$ yields
\begin{equation}
g(b) = 
\ln b \, \ln(1+b) + \int_{0}^{1} \frac{\ln s \, ds}{1+ bs}. 
\end{equation}
\noindent
Expanding the integrand in a geometric series yields the final identity. 
\end{proof}

\begin{Thm}
\label{thm2}
Let $b, \, r > 0$. Then
\begin{equation}
\int_{0}^{b} \frac{\ln x \, dx}{x+r} = \ln b \ln \left( \frac{b+r}{r} \right) 
+ \text{Li}_{2} \left( - \frac{b}{r} \right). 
\end{equation}
\end{Thm}

\begin{Cor}
Let $b > 0$. Then 
\begin{equation}
\int_{0}^{b} \frac{\ln x \, dx }{x+b} = \ln 2 \, \ln b - \frac{\pi^{2}}{12}.
\end{equation}
\end{Cor}
\begin{proof}
Let $r \to b$ in Theorem \ref{thm2} and use 
\begin{equation}
\text{Li}_{2}(-1) = \sum_{n=1}^{\infty} \frac{(-1)^{n}}{n^{2}} = 
- \frac{\pi^{2}}{12}.
\end{equation}
\end{proof}

\medskip

The expression in Theorem \ref{thm2} and the method of partial fractions 
gives the explicit evaluation of elementary logarithmic integrals where 
the rational function has simple poles. For example: 

\begin{Cor}
Let $0 <a  < b$ and $r_{1} \neq r_{2} \in \mathbb{R}^{+}$. Then, with 
$r = r_{2}-r_{1}$, we have
\begin{eqnarray}
\int_{a}^{b} \frac{\ln x \, dx}{(x+r_{1})(x+r_{2})} & = &  
\frac{1}{r} 
\left[ \ln b \, \ln \left( \frac{r_{2}(b+r_{1})}{r_{1}(b+r_{2})} \right) + 
\ln a \, \ln \left( \frac{r_{1}(a+r_{2})}{r_{2}(a+r_{1})} \right) \right] +
\nonumber \\
& + & \frac{1}{r} \left[ 
\text{Li}_{2} \left( - \frac{b}{r_{1}} \right) - 
\text{Li}_{2} \left( - \frac{a}{r_{1}} \right) - 
\text{Li}_{2} \left( - \frac{b}{r_{2}} \right) + 
\text{Li}_{2} \left( - \frac{a}{r_{2}} \right) \right]. \nonumber
\end{eqnarray}
\end{Cor}

The special case $a = r_{1}$ and $b = r_{2}$ is of interest:

\begin{Cor}
\label{cor3}
Let $0 <a  < b$. Then 
\begin{eqnarray}
\int_{a}^{b} \frac{\ln x \, dx}{(x+a)(x+b)}  & = & 
\frac{1}{b-a} \left[ \ln(ab) \ln(a+b) -\ln2 \ln(ab) - 
2 \ln a \, \ln b \right] \nonumber \\
&+  &  \frac{1}{b-a} \left[ - 2\text{Li}_{2}(-1)
+ \text{Li}_{2} \left( - \frac{b}{a} \right) + 
\text{Li}_{2} \left( - \frac{a}{b} \right) \right]. 
\nonumber
\end{eqnarray}
\end{Cor}

\medskip

The integral in Corollary \ref{cor3} appears as $\mathbf{4.232.1}$ in \cite{gr}.  
An interesting problem is to derive $\mathbf{4.232.2}$
\begin{equation}
\ift \frac{\ln x \, dx}{(x+u)(x+v)} = \frac{\ln^{2}u - \ln^{2}v}{2(u-v)}
\end{equation}
\noindent
directly from Corollary  \ref{cor3}.  \\

We now present an elementary evaluation of this integral and obtain from it 
an identity of Euler.   We prove that 
\begin{equation}
\int_{a}^{b} \frac{\ln x \, dx}{(x+a)(x+b)} = 
\frac{\ln ab}{2(b-a)} \ln \frac{(a+b)^2}{4ab}. 
\end{equation}

\no
\begin{proof}
The partial fraction decomposition
\ba
\frac{1}{(x+a)(x+b)} & = & \frac{1}{b-a} \left( \frac{1}{x+a} - \frac{1}{x+b}
\right). \nn
\ea
\no
reduces the problem to the evaluation of 
\begin{equation}
I_{1}  =  \int_{a}^{b} \frac{\ln x \; dx}{x+a} \text{ and } 
I_{2}  =  \int_{a}^{b} \frac{\ln x \; dx}{x+b}. \nn 
\end{equation}

\medskip

The change of variables $x = at$ gives, with $c = b/a$,  
\ba
I_{1} & = & \int_{1}^{c} \frac{\ln(at) \; dt}{1+t} \nn \\
     & = & \ln a \int_{1}^{c} \frac{dt}{1+t} + \int_{1}^{c} \frac{\ln t}{1+t}
\, dt \nn \\
& = & \ln a \ln(1 + c) - \ln a \ln 2 + \int_{1}^{c} \frac{\ln t}{1+t} \, dt. 
\nn
\ea
\no
Similarly,
\ba
I_{2} & = & \ln b \ln 2 - \ln b \ln (1 + 1/c) + \int_{1}^{1/c} \frac{\ln t}
{1+t} \, dt. \nn
\ea
\no
Therefore 
\ba
I_{1}-I_{2} & = & \ln a \ln(1+ c) + \ln b \ln(1+1/c) - \ln 2 \ln a - \ln 2 
\ln b + \nn \\
 & + & \int_{1}^{c} \frac{\ln t}{1+t} \, dt - \int_{1/c}^{1} \frac{\ln t}{1+t}
\, dt. \nn
\ea
\no
Let $s = 1/t$ in the second integral to get
\ba
\int_{1/c}^{1} \frac{\ln t}{1+t} \, dt & = & 
\int_{c}^{1} \frac{\ln s}{s(1+s)} \, ds. \nn
\ea
\no
Replacing in the expression for $I_{1}-I_{2}$ yields
\ba
I_{1}-I_{2} & = & \ln a \left( \ln (a+b) -\ln a - \ln 2 \right) -
 \ln b \left( \ln 2 - \ln (a+b) + \ln b \right) + \nn \\
 & + & \int_{1}^{c} \frac{\ln t}{t} \, dt. \nn
\ea
\no
The last integral can now be evaluated by elementary means to 
produced the result. 
\end{proof}

Now comparing the two evaluation of the integral in Corollary \ref{cor3} 
produces an identity for the dilogarithm function. 

\begin{Cor}
The dilogarithm function satisfies
\begin{equation}
\text{Li}_{2}(-z) + \text{Li}_{2} \left(- \frac{1}{z} \right) 
= -\frac{\pi^{2}}{6} - \frac{1}{2} \ln^{2}(z).
\end{equation}
\end{Cor}

This is the first of many interesting functional equations satisfied by the
polylogarithm functions. It was established by L. Euler in $1768$. The 
reader will find in \cite{lewin1} a nice 
description of them. 

\section{A single multiple pole} \label{sec-multiple} 
\setcounter{equation}{0}

In this section we consider the evaluation of 
\begin{equation}
f_{n}(b,r) := \int_{0}^{b} \frac{\ln x \, dx}{(x+r)^{n}}. 
\end{equation}
\noindent
This corresponds to the elementary rational integrals with a single pole 
(at $x = -r)$. The change of variables $x = rt$ yields
\begin{equation}
f_{n}(b,r) = \frac{\ln r}{(n-1)r^{n-1}} 
\left[ \frac{(b+r)^{n-1} - r^{n-1}}{(b+r)^{n-1}} \right] + 
\frac{1}{r^{n-1}} h_{n}(b/r), 
\nonumber
\end{equation}
\noindent
where
\begin{equation}
h_{n}(b) := \int_{0}^{b} \frac{\ln t \, dt}{(1+t)^{n}}. 
\end{equation}

\medskip

We first establish a recurrence for $h_{n}$. \\

\begin{Thm}
Let $n > 2$ and $b > 0$. Then $h_{n}$ satisfies the recurrence
\begin{equation}
h_{n}(b) = \frac{n-2}{n-1} h_{n-1}(b) + 
\frac{b \, \ln b }{(n-1)(1+b)^{n-1}} + 
\frac{ 1 - (1+b)^{n-2}}{(n-1)(n-2) (1+b)^{n-2}}. 
\end{equation}
\end{Thm}
\begin{proof}
Start with 
\begin{equation}
h_{n}(b)  =  \int_{0}^{b} \frac{[ (1+t)-t] \, \ln t \, dt}{(1+t)^{n} }
=  h_{n-1}(b) - \int_{0}^{b} \frac{t \, \ln t \, dt}{(1+t)^{n}}. 
\nonumber 
\end{equation}
\noindent
Integrate by parts in the last integral, with $u = t \, \ln t$ and 
$dv = dt/(1+t)^{n}$ to produce the result. 
\end{proof}

\medskip

The initial condition for this recurrence is obtained from the value
\begin{equation}
h_{2}(b) = \frac{b}{1+b} \ln b - \ln(1+b). 
\end{equation}
\noindent
This expression follows by a direct integration by parts in
\begin{equation}
h_{2}(b) = - \lim\limits_{\epsilon \to 0} \int_{\epsilon}^{b} 
\ln t \, \frac{d}{dt} (1+t)^{-1} \, dt. 
\end{equation}

\medskip

The first few values of $h_{n}(b)$ suggest the introduction of the function
\begin{equation}
q_{n}(b) := (1+b)^{n-1}h_{n}(b),
\end{equation}
\noindent
for $n \geq 2$. For example, 
\begin{equation}
q_{2}(b) = b \ln b - (1+b) \, \ln(1+b). 
\end{equation}

The recurrence for $h_{n}$ yields one for $q_{n}$.

\begin{Cor}
\label{cor2}
The  recurrence
\begin{equation}
q_{n}(b) = \frac{(n-2)}{(n-1)}(1+b) q_{n-1}(b) + 
\frac{b \, \ln b}{n-1} - \frac{(1+b) \left[ (1+b)^{n-2}-1 \right]}{(n-1)(n-2)},
\label{recurq}
\end{equation}
\noindent
holds for $n \geq 2$.
\end{Cor}

\medskip

Corollary \ref{cor2} establishes the existence of functions $X_{n}(b), \, 
Y_{n}(b)$ and $Z_{n}(b)$, such that
\begin{equation}
q_{n}(b) = X_{n}(b) \, \ln b + Y_{n}(b) \, \ln(1+b) + Z_{n}(b). 
\end{equation}
\noindent
The recurrence (\ref{recurq}) produces explicit expression for  each of these
parts. 

\begin{Prop}
Let $n \geq 2$ and $b > 0$. Then
\begin{equation}
X_{n}(b) = \frac{(1+b)^{n-1}-1}{n-1}. 
\end{equation}
\end{Prop}
\begin{proof}
The function $X_{n}$ satisfies the recurrence
\begin{equation}
X_{n}(b)  =  \frac{n-2}{n-1}(1+b)X_{n-1}(b) + \frac{b}{n-1}. \label{recX}
\end{equation}
\noindent
The initial condition is $X_{2}(b) = b$. The result is now easily established
by induction.
\end{proof}

\begin{Prop}
Let $n \geq 2$ and $b > 0$. Then
\begin{equation}
Y_{n}(b) = -\frac{(1+b)^{n-1}}{n-1}. 
\end{equation}
\end{Prop}
\begin{proof}
The function $Y_{n}$ satisfies the recurrence
\begin{equation}
Y_{n}(b)  =  \frac{n-2}{n-1}(1+b)Y_{n-1}(b). \label{recY}
\end{equation}
\noindent
This recurrence and the initial condition $Y_{2}(b) = -(1+b)$, yield the result.
\end{proof}

\medskip 

It remains to identify the function $Z_{n}(b)$. It satisfies the 
recurrence
\begin{equation}
Z_{n}(b)  =  \frac{n-2}{n-1}(1+b) Z_{n-1}(b) - 
\frac{(1+b) \left[ (1+b)^{n-2}-1 \right]}{(n-2)(n-1)}. \label{recZ} 
\end{equation}
This recurrence and the initial condition $Z_{2}(b) = 0$ suggest the 
definition
\begin{equation}
T_{n}(b) := - \frac{(n-1)! \, Z_{n}(b)}{b(1+b)}. 
\end{equation}

\medskip

\begin{Lem}
The function $T_{n}(b)$ is a polynomial of degree $n-3$ with positive 
integer coefficients.
\end{Lem}
\begin{proof}
The function $T_{n}(b)$ satisfies the recurrence
\begin{equation}
T_{n}(b) = (n-2)(1+b)T_{n-1}(b) + (n-3)! 
\left[ \frac{(1+b)^{n-2}-1}{b} \right]. 
\end{equation}
\noindent
Now simply observe that the right hand side is a polynomial in $b$.
\end{proof}

Properties of the polynomial $T_{n}(b)$ will be described in future 
publications. We now simply observe that its coefficients are {\em unimodal}.
Recall that a polynomial 
\begin{equation}
P_{n}(b) = \sum_{k=0}^{n} c_{k}b^{k} 
\end{equation}
\noindent
is called {\em unimodal} if there is an  index $n^{*}$, such that 
$c_{k} \leq c_{k+1}$ for $0 \leq k \leq n^{*}$ and 
$c_{k} \geq c_{k+1}$ for $n^{*} < k \leq n$. That is, the sequence of 
coefficients of $P_{n}$ has a single peak. Unimodal polynomials appear in 
many different branches of Mathematics. The reader will find in 
\cite{brenti1} and  \cite{stanley1} information about this property. We now 
use the result of \cite{bomouni1} to establish the unimodality of  $T_{n}$. 

\begin{Thm}
Suppose $c_{k} > 0$ is a nondecreasing sequence. Then $P(x+1)$ is unimodal.
\end{Thm}

Therefore we consider the polynomial $S_{n}(b) := T_{n}(b-1)$. It satisfies 
the recurrence
\begin{equation}
S_{n}(b) = b(n-2)S_{n-1}(b) + (n-3)! \, \sum_{r=0}^{n-3}b^{r}. 
\end{equation}
\noindent
Now write 
\begin{equation}
S_{n}(b) = \sum_{k=0}^{n-3} c_{k,n}b^{k},
\end{equation}
\noindent
and conclude that $c_{0,n} = (n-3)!$ and 
\begin{equation}
c_{k,n} = (n-2)c_{k-1,n-1} + (n-3)!,
\end{equation}
\noindent
from which it follows that
\begin{equation}
c_{k+1,n}-c_{k,n} = (n-2) \left[ c_{k,n-1}-c_{k-1,n-1} \right]. 
\end{equation}
\noindent
We conclude that $c_{k,n}$ is a nondecreasing sequence.

\medskip

\begin{Thm}
The polynomial $T_{n}(b)$ is unimodal.
\end{Thm}

\medskip

\noindent
{\bf Conclusions}. We have given explicit formulas for integrals of the 
form
\begin{equation}
\int_{a}^{b} R(x) \, \ln x \, dx,
\end{equation}
\noindent
where $R$ is a rational function with real poles. Future reports will 
describe the case of higher powers
\begin{equation}
\int_{a}^{b} R(x) \, \ln^{m} x \, dx, 
\end{equation}
\noindent
as well as the case of complex poles,
based on integrals of the form
\begin{equation}
C_{n}(a,r) := \int_{0}^{b} \frac{\ln x \, dx}{(x^{2}+r^{2})^{n}}.
\end{equation}

\bigskip

\noindent
{\bf Acknowledgments}. The author wishes to thank Luis Medina for a careful
reading of the manuscript. The partial support of 
$\text{NSF-DMS } 0409968$ is also acknowledged.


\bigskip

\begin{thebibliography}{99}
\bibitem{bomouni1}
G. Boros and V. Moll: {\em A criterion for unimodality}. Elec. Journal of 
Combinatorics, {\bf 6}, 1999, 1-6. 

\bibitem{brenti1}
F. Brenti: {\em Log-concave and unimodal sequences in Algebra, Combinatorics
and Geometry: an update}. Contemporary Mathematics, {\bf 178}, 1994, 71-89.

\bibitem{gr}
I. S. Gradshteyn and I. M. Rhyzik: {\em Table of Integrals, Series, and 
Products}, 6th edition, 2000. Edited by A. Jeffrey and D. Zwillinger. 
Academic Press, New York. 

\bibitem{lewin1}
L. Lewin: {\em Dilogarithms and Associated Functions}, 2nd edition, 1981. 
Elsevier, North Holland, Amsterdam, The Netherlands. 

\bibitem{moll-gr1}
V. Moll: {\em The integrals in Grdashteyn and Rhyzik. Part 1: a family of 
logarithmic integrals}. Scientia, {\bf 14}, 2007, 1-6.

\bibitem{stanley1}
R. Stanley: {\em Log-concave and unimodal sequences in Algebra, Combinatorics
and Geometry}. Graph Theory and its applications: East and West Jinan, 1986. 
Annals of the New York Academy of Sciences, {\bf 576}, 1989, 500-535. 

\end{thebibliography}
\end{document}